\documentclass{amsart}
\usepackage{dsfont}
\usepackage{enumerate}
\usepackage{mathrsfs}
\usepackage[all]{xy}

\newtheorem{theorem}{Theorem}[section]
\newtheorem{lemma}[theorem]{Lemma}
\newtheorem{proposition}[theorem]{Proposition}
\newtheorem{corollary}[theorem]{Corollary}

\theoremstyle{definition}
\newtheorem{definition}[theorem]{Definition}

\theoremstyle{remark}
\newtheorem{remark}[theorem]{Remark}
\newtheorem*{ack*}{Acknowledgement}

\numberwithin{equation}{section}

\begin{document}

\title[On $C^*$-algebras related to constrained representations]{On $C^*$-algebras related to constrained representations of a free
group}

\author{V. M. Manuilov}
\address{V. M. Manuilov, Department of Mechanics and Mathematics, Moscow State University, Leninskie Gory, Moscow, 119991, Russia, and Academy of Fundamental and Interdisciplinary Science, Harbin Institute of Technology,
Harbin, 150001, P.R.China}
\email{manuilov@mech.math.msu.su}
\thanks{The first named author acknowledges partial support from RFFI, grant No. 08-01-00034}

\author{Chao You}
\address{Chao You, Department of Mathematics and Academy of Fundamental and Interdisciplinary Science, Harbin Institute of Technology, Harbin, 150001, P.R.China}
\email{hityou1982@gmail.com}
\thanks{The second named author is partially supported by National Natural Science Foundation of China (No.10971150) and State Scholarship Fund of China (No.2008612056).}





\begin{abstract}
We consider representations of the free group $F_2$ on two
generators such that the norm of the sum of the generators and
their inverses is bounded by $\mu\in[0,4]$. These
$\mu$-constrained representations determine a $C^*$-algebra
$A_{\mu}$ for each $\mu\in[0,4]$. When $\mu=4$ this is the full
group $C^*$-algebra of $F_2$. We prove that these $C^*$-algebras
form a continuous bundle of $C^*$-algebras over $[0,4]$ and
calculate their $K$-groups.
\end{abstract}

\maketitle

\section{Introduction}

The aim of this paper is to study certain family of $C^*$-algebras
related to representations of a free group with a given bound for
the norm of the sum of the generating elements.

Let $\Gamma$ be a discrete group. If we consider different sets of
unitary representations of $\Gamma$ (all representations in this
paper are unitary ones), they lead to different group
$C^*$-algebras of $\Gamma$. For example, the full group
$C^*$-algebra of $\Gamma$, denoted by $C^*(\Gamma)$, is the
closure of the group ring $\mathds{C}[\Gamma]$ with respect to the
norm induced by the universal representation (or, equivalently, by
all representations); while the reduced group $C^*$-algebra of
$\Gamma$, denoted by $C^*_r(\Gamma)$, is the closure of
$\mathds{C}[\Gamma]$ with respect to the norm induced by the
regular representation. Here we consider some special classes of
representations for the free group on two generators in order to
obtain the corresponding $C^*$-algebras. These classes are related
to the special element $x$ of the group ring --- the sum of all
generators and their inverses, sometimes called an {\it averaging
operator}. This element plays an important role in research
related to groups and their $C^*$-algebras. For example,
amenability of $\Gamma$ is equivalent to $\|\lambda(x)\|=n$, where
$\lambda$ is the regular representation of $\Gamma$ (we use the
same notation for representations of groups and of their group
rings and $C^*$-algebras) and $n$ is the number of summands in $x$
(twice the number of generators). Property (T) for $\Gamma$ is
equivalent to existence of a spectral gap near $n$ in the spectrum
of $\pi(x)$ for the universal representation $\pi$
\cite{delaHarpe}.

Let $u$ and $v$ denote the two generators of the free group $F_2$.
Then $x=u+u^{-1}+v+v^{-1}\in \mathds{C}[\Gamma]$.

\begin{definition}
For $\mu\in[0,4]$, a representation $\pi:F_2\rightarrow
\mathcal{U}(\mathcal{H}_{\pi})$ is called a
\emph{$\mu$-constrained representation} if
\begin{equation}
\|\pi(x)\|=\|\pi(u)+\pi(u)^*+\pi(v)+\pi(v)^*\|\leq \mu.
\end{equation}
\end{definition}

Given $0\leq\mu\leq4$, the assignment $u,v\mapsto (\mu/4\pm
i\sqrt{1-(\mu/4)^2})I$, where $I$ is the identity operator on any
(in particular, one-dimensional) Hilbert space, gives rise to a
$\mu$-constrained representation of $F_2$ for any $\mu\in[0,4]$.
This shows that $\mu$-constrained representations exist. Actually
there are abundant. For example, all representations of $F_2$ are
$4$-constrained ones, in which case there is actually no
constraint at all. Moreover, if $\pi$ is a $\mu$-constrained
representation and $\mu \leq \mu' \leq4$, then $\pi$ is also a
$\mu'$-constrained one. Let $\Pi_{\mu}$ denote the set of all
$\mu$-constrained representations, then $\Pi_{\mu_1}\subseteq
\Pi_{\mu_2}$ if $0\leq \mu_1 \leq \mu_2\leq 4$. The
one-dimensional example above shows also that if $\mu_1\neq\mu_2$
then $\Pi_{\mu_1}$ is strictly smaller than $\Pi_{\mu_2}$. Note
that $\Pi_4$ consists of all representations of $F_2$.

As in the case of $C^*(F_2)$, we first define a (semi)norm
$\|\cdot\|_{\mu}$ over $\mathds{C}[F_2]$ induced by $\Pi_{\mu}$
and then complete $\mathds{C}[F_2]$ with respect to
$\|\cdot\|_{\mu}$, thus obtaining the corresponding group
$C^*$-algebra $A_{\mu}$.

\begin{definition}
For $a\in \mathds{C}[F_2]$, $\mu\in[0,4]$, set
\begin{equation}
\|a\|_{\mu}:=\sup_{\pi\in\Pi_{\mu}}\|\pi(a)\|.
\end{equation}
\end{definition}

\begin{remark}
Since $\|a\|_\mu\leq\|a\|_4= \|a\|_{\text{max}}$, where $\|\cdot
\|_{\text{max}}$ is the norm on $\mathds{C}[F_2]$ induced by the
universal representation, it is clear that $\|\cdot\|_{\mu}$ is
bounded. Moreover, $\|\cdot\|_{\mu_1}\leq \|\cdot\|_{\mu_2}$ if
$0\leq \mu_1 \leq \mu_2\leq 4$. As we use only unitary
representations, this is a $C^*$-seminorm.

\end{remark}

Set $\mathcal{N}_\mu=\{a\in\mathds{C}[F_2]:\|a\|_\mu=0\}$ and complete
$\mathds{C}[F_2]/\mathcal{N}_\mu$ with respect to $\|\cdot\|_{\mu}$ (which
is already a norm there). Let us denote this completion by
$A_\mu$. It is obviously a $C^*$-algebra for any $\mu\in[0,4]$.
Our aim is to study the family of $C^*$-algebras $A_\mu$.

\begin{remark}
Note that $A_\mu$ can be defined as a {\it universal}
$C^*$-algebra generated by two unitaries, $u$, $v$ satisfying a
single relation $\|u+u^*+v+v^*\|\leq\mu$.

\end{remark}

\begin{proposition}\label{quotient}
For any $0\leq \mu_1 \leq \mu_2\leq 4$, the identity map on
$\mathds{C}[F_2]$ extends to a surjective $*$-homomorphism
$A_{\mu_2}\to A_{\mu_1}$.

\end{proposition}

\begin{proof}
Since $\mathcal{N}_{\mu_2}\subset \mathcal{N}_{\mu_1}$, the identity map on
$\mathds{C}[F_2]$ gives rise to a map
$\mathds{C}[F_2]/\mathcal{N}_{\mu_2}\to\mathds{C}[F_2]/\mathcal{N}_{\mu_1}$, which
extends to a $*$-homomorphism from $A_{\mu_2}$ to $A_{\mu_1}$ by
continuity. Since the range of this $*$-homomorphism is dense in
$A_{\mu_1}$, it is surjective.

\end{proof}

Note that $A_4$ is isomorphic to the full group $C^*$-algebra
$C^*(F_2)$. Later on we shall give a description for $A_0$. For
$2\sqrt{3}\leq\mu\leq 4$ the identity map on $\mathds{C}[F_2]$
extends to a surjective $*$-homomorphism from $A_\mu$ to the
reduced group $C^*$-algebra $C^*_r(F_2)$ \cite{Kesten}.

The aim of this paper is to study the family of $C^*$-algebras
$A_\mu$. In the next section we show that this family is a
continuous bundle of $C^*$-algebras and then we identify $A_0$ as
a certain amalgamated free product. Finally, following Cuntz
\cite{Cuntz}, we calculate the $K$-theory groups for $A_\mu$ and
show that they don't depend on $\mu$.

\section{Continuity of $A_\mu$}

If $\mu_1$ is close to $\mu_2$ then one would expect that
$A_{\mu_1}$ and $A_{\mu_2}$ are close to each other. In other
words, there is some kind of ``continuity'' of $A_{\mu}$ with
respect to $\mu$. In order to characterize such ``continuity'', we
use the notion of continuous bundle of $C^*$-algebras due to
Dixmier.

Let $I$ be a locally compact Hausdorff space and let
$\{A(x)\}_{x\in I}$ be a family of $C^*$-algebras. Denote by
$\prod_{x\in I}A(x)$ the set of functions $a=a(x)$ defined on $I$
and such that $a(x)\in A(x)$ for any $x\in I$.

\begin{definition}[\cite{Dixmier}]
Let $\mathcal{A}\subset \prod_{x\in I}A(x)$ be a subset with the
following properties:

\begin{enumerate}[(i)]

\item $\mathcal{A}$ is a $*$-subalgebra in $\prod_{x\in I}A(x)$,

\item for any $x\in I$ the set $\{a(x):a\in \mathcal{A}\}$ is dense in the algebra $A(x)$,

\item for any $a\in \mathcal{A}$ the function $x\mapsto \|a(x)\|$ is continuous,

\item let $a\in \prod_{x\in I}A(x)$, if for any $x\in I$ and for any $\varepsilon>0$ one can find such
$a'\in \mathcal{A}$ such that $\|a(x)-a'(x)\|<\varepsilon$ in some neighborhood of the point $x$, then one has $a\in\mathcal{A}$.
\end{enumerate}
Then the triple $(A(x),I,\mathcal{A})$ is called a continuous
bundle of $C^*$-algebras.
\end{definition}

Let $I=[0,4]$, $A=C(I,C^*(F_2))$ and let $B=\{f\in
A:\|f(\mu)\|_{\mu}=0,\forall \mu\in I\}$. It is clear that $B$ is
a closed ideal of $A$, with the quotient map $q:A\rightarrow A/B$.
Define the map $\iota :A/B\rightarrow \prod_{\mu\in I}A_{\mu}$ by
$\iota(b)(\mu)=q_{\mu}(a(\mu))$, where $b\in A/B$ and $a\in A$
such that $b=q(a)$, $q_{\mu}:C^*(F_2)=A_4\rightarrow A_{\mu}$ is
the quotient map. It is simple to check that $\iota$ is
well-defined and injective, so from now on we treat $A/B$ as a
subalgebra of $\prod_{\mu\in I}A_{\mu}$. In order to prove that
$(A_{\mu},I,A/B)$ is a continuous bundle of $C^*$-algebras, we
need some lemmas.

\begin{lemma}\label{norm continuity}
For any $a\in C^*(F_2)$, the function $N_a:I\rightarrow
\mathds{R}_+$ defined by $\mu\mapsto \|a\|_{\mu}$ is continuous.
\end{lemma}

\begin{proof}
Given any fixed $\mu_0\in I$, we will prove that $N_a$ is
continuous at $\mu_0$ in two steps: $N_a$ is left and right
continuous at $\mu_0$, respectively.

\textbf{Step 1.} Note that $N_a$ is a non-decreasing function, so
$l=\lim_{\mu\rightarrow \mu_0^-}N_a(\mu)$ exists. Assume that
$l<N_a(\mu_0)$, then there must exist a representation $\pi$ of
$F_2$ such that $\|\pi(u+u^*+v+v^*)\|=\mu_0$ and $l<\|\pi(a)\|\leq
N_a$.

Let us first give a family of Borel functions $\{f_t:S^1\rightarrow S^1\}_{t\in[0,1]}$ as follows:
\begin{equation*}
f_t(e^{i\theta})=
\begin{cases}
e^{i\arccos((1-t)\cos\theta)},& \theta\in [0,\pi]\\
e^{-i\arccos((1-t)\cos\theta)}, & \theta\in (-\pi,0)
\end{cases}
\end{equation*}
Applying Borel functional calculus of $f_t$ to $\pi(u)$ and
$\pi(v)$, we get a new representation $\pi_t$ of $F_2$ which is
defined by $u\mapsto f_t(\pi(u))$ and $v\mapsto f_t(\pi(v))$, and
$\{\pi_t\}_{t\in [0,1]}$ is a continuous family of
representations. Since
$f_t(z)+f_t(\overline{z})=(1-t)(z+\overline{z})$, we have
$\pi_t(u+u^*+v+v^*)=f_t(\pi(u))+f_t(\pi(u^*))+f_t(\pi(v))+f_t(\pi(v^*))=(1-t)(u+u^*+v+v^*)$.
Hence $\|\pi_t(u+u^*+v+v^*)\|<\mu_0$. Meanwhile, $\|\pi_t(a)\|$
varies also continuously, which contradicts the assumption.

\textbf{Step 2.} Assume that, for some $a\in\mathbb C[F_2]$, $N_a$
is not continuous at $\mu_0$ from the right, i.e.,
$N_a(\mu_0)<\lim_{\mu\rightarrow \mu_0^+}N_a(\mu)=r$. Then there
must exist a family of representations $\{\pi_n:F_2\rightarrow
\mathcal{U}(\mathcal{H}_n)\}_{n\in\mathds{N}}$ such that
$\{\|\pi_n(u+u^*+v+v^*)\|\}_{n\in\mathds{N}}$ is a decreasing
sequence convergent to $\mu_0$ and
$\lim_{n\rightarrow\infty}\|\pi_n(a)\|=r$.

Let $\prod_{n\in\mathds{N}}B(\mathcal{H}_n)$ be the $C^*$-algebra
of all sequences $b=(b_1,b_2,\cdots)$, $b_n\in B(\mathcal{H}_n)$,
such that $\|b\|:=\sup_{n\in\mathds{N}}\|b_n\|<\infty$. Let
$\oplus_{n\in\mathds{N}}B(\mathcal{H}_n)$ be the ideal of
$\prod_{n\in\mathds{N}}B(\mathcal{H}_n)$ that consists of
sequences $(b_1,b_2,\cdots)$ such that
$\lim_{n\rightarrow\infty}\|b_n\|=0$. Then
$\prod_{n\in\mathds{N}}B(\mathcal{H}_n)/\oplus_{n\in\mathds{N}}B(\mathcal{H}_n)$
is a quotient $C^*$-algebra. By Gelfand-Naimark-Segal theorem,
there exists a faithful representation
$\rho:\prod_{n\in\mathds{N}}B(\mathcal{H}_n)/\oplus_{n\in\mathds{N}}B(\mathcal{H}_n)\to
B(\mathscr{H})$ for some Hilbert space $\mathscr{H}$.  Let $q:
\prod_{n\in\mathds{N}}B(\mathcal{H}_n)\rightarrow
\prod_{n\in\mathds{N}}B(\mathcal{H}_n)/\oplus_{n\in\mathds{N}}B(\mathcal{H}_n)$
be the canonical quotient map. Note that, if
$b=(b_1,b_2,\cdots)\in \prod_{n\in\mathds{N}}B(\mathcal{H}_n)$,
$\|(\rho\circ
q)(b)\|=\|q(b)\|=\lim\sup_{n\rightarrow\infty}\|b_n\|$.

Let $\pi_\infty$ be the representation of $F_2$ defined by
$u\mapsto (\rho\circ q)((\pi_1(u),\pi_2(u),\cdots))$ and $v\mapsto
(\rho\circ q)((\pi_1(v),\pi_2(v),\cdots))$. Then
$\|\pi_\infty(u+u^*+v+v^*)\|=\lim\sup_{n\rightarrow\infty}\|\pi_n(u+u^*+v+v^*)\|=\mu_0$,
so $\pi_\infty$ is a $\mu_0$-constrained representation of $F_2$.
But
$\|\pi_\infty(a)\|=\lim\sup_{n\rightarrow\infty}\|\pi_n(a)\|=r>\|a\|_{\mu_0}$,
which is a contradiction.

\end{proof}

Recall that $B=\{f\in C(I,C^*(F_2)):\|f(\mu)\|_\mu=0 \ \ \text{for
any}\ \ \mu\in I\}$.

\begin{lemma}\label{ideal}
Set $I_{\mu}=\{a\in C^*(F_2):\|a\|_{\mu}=0\}$. Then
$\{g(\mu_0):g\in B\}=I_{\mu_0}$ for any $\mu_0\in I$.
\end{lemma}

\begin{proof}
From the definition of $B$, it is obvious that $\{f(\mu_0):f\in
B\}\subseteq I_{\mu_0}$, thus we just need to prove the converse
inclusion. An easy observation implies that it suffices to prove
this inclusion for positive elements of $I_{\mu_0}$.

Let $a\in I_{\mu_0}$ be positive. We have to find $g\in B$ such
that $g(\mu_0)=a$. Define a family of continuous functions by
\begin{equation*}
f_{\mu}(t)=
\begin{cases}
0,& t\in (-\infty,\|a\|_{\mu}]\\
t-\|a\|_{\mu},& t\in (\|a\|_{\mu},\infty).
\end{cases}
\end{equation*}

As $\|a\|_\mu=0$ for $\mu\leq \mu_0$, so $f_\mu(a)=a$ for $\mu\leq
\mu_0$. It follows from Lemma \ref{norm continuity} that $f_\mu$
is continuous in $\mu$. Define a function $g:I\to C^*(F_2)$ by
$g(\mu)=f_\mu(a)$. Then $g\in A=C(I,C^*(F_2))$ and $g(\mu_0)=a\in
I_{\mu_0}$.

Let $q_{\mu}:C^*(F_2)\rightarrow C^*(F_2)/I_{\mu}\cong A_{\mu}$
denote the quotient map. As $\|q_\mu(a)\|=\|a\|_\mu$, so
$q_{\mu}(f_{\mu}(a))=f_{\mu}(q_{\mu}(a))=0$, thus
$g(\mu)=f_{\mu}(a)\in I_{\mu}$, hence $g\in B$.

\end{proof}

Since we have treated $A/B$ as a subalgebra of $\prod_{\mu\in
I}A_{\mu}$, for any $b\in A/B$, besides the quotient norm, we can
also treat $b$ as a function defined on $I$ and take the supremum
norm. The following lemma asserts that these two norms coincide.

\begin{lemma}
Let $a\in A$, $b=q(a)\in A/B$. Set
\begin{equation*}
\|b\|_1=\inf_{g\in B}\|a+g\|=\inf_{g\in B}\sup_{\mu\in I}\|a(\mu)+g(\mu)\|
\end{equation*}
and
\begin{equation*}
\|b\|_2=\sup_{\mu\in I}\inf_{g\in B}\|a(\mu)+g(\mu)\|=\sup_{\mu\in
I}\|a(\mu)\|_{\mu}\text{ (by Lemma \ref{ideal})}.
\end{equation*}
Then $\|b\|_1=\|b\|_2$.
\end{lemma}

\begin{proof}
This follows from uniqueness of a $C^*$-norm on the $C^*$-algebra
$A/B$.
\end{proof}

\begin{theorem}
$(A_{\mu},I,A/B)$ is a continuous bundle of $C^*$-algebras.
\end{theorem}

\begin{proof}
Let us check the conditions from the definition of a continuous
bundle of $C^*$-algebras one by one.

(i) and (ii) are obviously satisfied and $\{a(\mu):a\in A/B\}$
equals $A_{\mu}$.

For any $b\in A/B$ with $b=q(a)$ where $a\in A$, given
$\mu,\mu'\in I$,
\begin{align*}
&|\|b(\mu')\|_{\mu'}-\|b(\mu)\|_{\mu}|\\
\leq&|\|a(\mu')\|_{\mu'}-\|a(\mu)\|_{\mu}|\\
\leq&|\|a(\mu')\|_{\mu'}-\|a(\mu)\|_{\mu'}|+|\|a(\mu)\|_{\mu'}-\|a(\mu)\|_{\mu}|\\
\leq&\|a(\mu')-a(\mu)\|_{\mu'}+|\|a(\mu)\|_{\mu'}-\|a(\mu)\|_{\mu}|\\
\leq&\|a(\mu')-a(\mu)\|_{\text{max}}+|\|a(\mu)\|_{\mu'}-\|a(\mu)\|_{\mu}|,
\end{align*}
If $\mu'$ is close to $\mu$ then $\|a(\mu')-a(\mu)\|_{\text{max}}$
is small because the function $\mu\mapsto a(\mu)$ is continuous,
and $\|a(\mu)\|_{\mu'}-\|a(\mu)\|_{\mu}$ is small due to Lemma
\ref{norm continuity}, therefore, the map
$\mu\mapsto\|b(\mu)\|_\mu$ is continuous, i.e., (iii) is
satisfied.

Suppose $z\in \prod_{\mu\in I}A_{\mu}$ such that for every $\mu\in
I$ and every $\varepsilon>0$, there exists an $b\in A/B$ such that
$\|z(\mu)-b(\mu)\|\leq\varepsilon$ in some neighborhood $U_\mu$ of
$\mu$. Thus we obtain an open covering $\{U_{\mu}\}_{\mu\in I}$ of
$I$. Let $\{U_i\}_{i=1}^{p}$ be its finite sub-covering and let
$(\eta_1,\ldots,\eta_p)$ be a continuous partition of unity in $I$
subordinate to the covering $\{U_i\}_{i=1}^{p}$. Then
$$
\|z(\mu)-\eta_1(\mu)b_1(\mu)-\cdots-\eta_p(\mu)b_p(\mu)\|\leq\varepsilon,\text{
for any $\mu\in I$},
$$
or equivalently,
$$\|z-\eta_1b_1-\cdots-\eta_pb_p\|\leq \varepsilon.$$
Since $\varepsilon>0$ is arbitrary, $\eta_ib_i$ belongs to $A/B$
and $A/B$ is norm closed, we have $z\in A/B$. So (iv) is
satisfied.

\end{proof}

\section{$A_0$ as an amalgamated free product}

Here we identify $A_0$ as an amalgamated product of
$C^*$-algebras.

Recall that, given $C^*$-algebras $A_1$, $A_2$ and $B$ and
embeddings $i_k :B\rightarrow A_k$, $k = 1, 2$, the
\emph{amalgamated free product} is the $C^*$-algebra, denoted
$A_1*_BA_2$, together with embeddings $j_k : A_k \rightarrow
A_1*_BA_2$, satisfying $j_1\circ i_1 = j_2 \circ i_2$, such that
the following holds: if $\phi_k : A_k\rightarrow A$, $k = 1, 2$,
are $*$-homomorphisms with $\phi_1\circ i_1 = \phi_2 \circ i_2$
then there is a unique $*$-homomorphism $\phi:
A_1*_BA_2\rightarrow A$ such that $\phi\circ j_k = \phi_k$, $k =
1,2$. The $*$-homomorphism $\phi$ induced by $\phi_1$ and $\phi_2$
will sometimes be denoted by $\phi_1*_B\phi_2$.

Let $p:S^1\to[-1,1]$ be the projection of the circle $x^2+y^2=1$
onto the $x$ axis. It induces an inclusion $i_1:C[-1,1]\to C(S^1)$
such that $i_1(\operatorname{id})=z+\overline{z}$, where $z=x+iy$
is the coordinate on $S^1$ and $\operatorname{id}$ is the identity
function on $C[-1,1]$. Let $\tau:C[-1,1]\to C[-1,1]$ be the flip
automorphism, which changes the orientation of the interval and is
given by $\operatorname{id}\mapsto -\operatorname{id}$. Set
$i_2=i_1\circ\tau$. Then
$i_2(\operatorname{id})=-(w+\overline{w})$.

The inclusions $i_1$ and $i_2$ of $C[-1,1]$ into $C(S^1)$ give us
the amalgamated free product $D=C(S^1)\ast_{C[-1,1]}C(S^1)$.

\begin{lemma}
$C^*$-algebras $A_0$ and $D$ are isomorphic.

\end{lemma}
\begin{proof}
Recall that $A_0$ is a universal $C^*$-algebra generated by two
unitaries, $u$ and $v$, with a single relation $u+u^*=-(v+v^*)$.

Let $\tilde u$, $\tilde v$ be generators for the two copies of
$C(S^1)$. Define $\varphi_k:C(S^1)\to A_0$, $k=1,2$, by
$\varphi_1(\tilde u)=u$, $\varphi_2(\tilde v)=v$. Then
$\varphi_1\circ i_1=\varphi_2\circ i_2$, hence the maps
$\varphi_k$ give rise to a $*$-homomorphism $D\to A_0$.

Using universality of $A_4$, we can construct a $*$-homomorphism
$\psi:A_4\to D$ by setting $\psi(u)=\tilde u\ast 1$,
$\psi(v)=1\ast\tilde v$. Note that $A_0$ is the quotient of $A_4$
under a single relation $u+u^*=-(v+v^*)$, and
$\psi(u+u^*)=-\psi(v+v^*)$, therefore, $\psi$ factorizes through
$A_0$, thus giving a $*$-homomorphism from $A_0$ to $D$.

The two $*$-homomorphisms $D\to A_0$ and $A_0\to D$ are obviously
inverse to each other, hence the two $C^*$-algebras are
isomorphic.

\end{proof}

Now we may apply the $K$-theory exact sequence for amalgamated
free products due to Cuntz \cite{Cuntz}:
\begin{equation*}
\begin{xymatrix}{
K_0(C[-1,1]) \ar[r]^-{({i_1},{i_2})} & K_0(C(S^1))\oplus K_0(C(S^1)) \ar[r]^-{{j_1}-{j_2}} & K_0(A_0) \ar[d]\\
K_1(A_0) \ar[u]&  K_1(C(S^1))\oplus K_1(C(S^1))
\ar[l]^-{{j_1}-{j_2}} & K_1(C[-1,1])\ar[l]^-{({i_1},{i_2})}}
\end{xymatrix}
\end{equation*}

\begin{corollary}\label{K-theory of A_0}
\begin{enumerate}[(i)]
\item $K_0(A_0)\cong \mathds{Z}$ and is generated by the class $[1]$ of unit element;\\
\item $K_1(A_0)\cong \mathds{Z}^2$ and is generated by $[u]$ and $[v]$,
which are considered as elements of the first and the second copy of $C(S^1)$ respectively.
\end{enumerate}
\end{corollary}

\section{$K$-Groups of $A_{\mu}$}

In \cite{Cuntz}, J. Cuntz proved that $K_0(C^*(F_2))\cong
\mathds{Z}$, $K_1(C^*(F_2))\cong \mathds{Z}^2$. Here we use his
method to calculate the $K$-groups for $A_{\mu}$, $0\leq\mu<4$.

\begin{remark} From Corollary \ref{K-theory of A_0} we can get
some information about $K$-groups of $A_{\mu}$($0<\mu<4$). Since
the quotient map $A_4\to A_0$ factorizes through $A_\mu$ and
induces an isomorphism in $K$-theory, we may conclude that
$K_*(A_\mu)$ contains $K_*(A_4)$ as a direct summand.

\end{remark}

In Section 1 we show that $A_{\mu}$ posseses certain continuity
with respect to $\mu$, together with the fact that the $K$-groups
of $A_0$ and $C^*(F_2)$ are the same, it would be reasonable to
conjecture that all $A_{\mu}$($0\leq \mu\leq 4$) have the
\emph{same} $K$-groups. Below we give a proof of this conjecture.
The idea of the proof is taken from \cite{Cuntz} (cf. Appendix in
\cite{Wassermann}): to construct a homotopy between the universal
representation of $F_2$ and the trivial representation. But the
trivial representation is not constrained for any $\mu<4$, so we
have to replace it by some other representation.

\begin{theorem}
The quotient map $A_4\to A_\mu$ induces an isomorphism of their
$K_*$-groups.

\end{theorem}

\begin{proof}
Let $B=C(S^1\vee S^1)$ be the $C^*$-algebra of continuous
functions on the wedge $S^1\vee S^1$ of two circles. This is the
algebra of pairs of functions $(f,g)$, $f,g \in C(S^1)$ such that
$f(1) = g(1)$, where $1\in S^1$ is the common point of the two
circles (we consider the circle as the subset of the complex plane
given by $|z|=1$). Then $K_0(B)\cong \mathds{Z}$, $K_1(B)\cong
\mathds{Z}^2$.


Set $\alpha(z)=-\operatorname{Re}z + i| \operatorname{Im} z|$.
Since $|\alpha(z)| = 1$, this is a function from $S^1$ to itself
with the trivial winding number (equivalently, the trivial
homotopy class). Note that $\operatorname{Re}(z + \alpha(z))$ = 0.

For each $\mu$ we define $*$-homomorphisms $\phi: A_{\mu}\rightarrow M_2(B)$ and $\psi : B \rightarrow M_2(A_{\mu})$ as follows,

Set $\psi: (z,1)\mapsto \begin{pmatrix} u & 0\\ 0 & 1\end{pmatrix}; (1,z)\mapsto \begin{pmatrix} 1 & 0\\ 0 & v\end{pmatrix}$.

Note that $\psi$ defines a $*$-homomorphism from $B$ to $M_2(A_{\mu})$ for $\mu= 4$, hence one can
pass to the quotient to obtain a $*$-homomorphism to $M_2(A_{\mu})$ for arbitrary $\mu$.

Set $\phi: u\mapsto \begin{pmatrix}(z,1) & (0,0)\\(0,0) & (-1,
\alpha(z))\end{pmatrix};v\mapsto \begin{pmatrix}(\alpha(z),-1)&
(0,0)\\(0,0) & (1,z)\end{pmatrix}$.

Note that $\phi(u+u^*+v+v^*)=\begin{pmatrix} (0,0) & (0,0)\\ (0,0)
& (0,0)\end{pmatrix}$, so $\phi$ is well-defined as a
$*$-homomorphism from $A_0$ to $M_2(B)$. Then it is well-defined
for any $\mu$.

For the composition $\psi \circ \phi: A_{\mu}\rightarrow
M_4(A_\mu)$, one has
$$
(\psi \circ \phi)(u)= \begin{pmatrix}u\\ & 1\\ & & -1\\& &
&\alpha(v)\end{pmatrix};\qquad (\psi \circ \phi)(v)=
\begin{pmatrix}\alpha(u)\\ & -1\\ & & 1\\& & &v\end{pmatrix}.
$$

Set
$$
V_t=\begin{pmatrix}\cos t & 0& 0 & \sin t\\0 & 1&  0& 0\\0& 0& 1&
0\\-\sin t &0& 0 & \cos
t\end{pmatrix}\begin{pmatrix}\alpha(u)\\&-1\\&&1\\&&&v\end{pmatrix}\begin{pmatrix}\cos
t & 0& 0 & -\sin t\\0 & 1&  0& 0\\0& 0& 1& 0\\\sin t &0& 0 & \cos
t\end{pmatrix},
$$
$t\in[0,\pi/2]$. Then one can define a homotopy of
$*$-homomorphisms $\lambda_t:A_{\mu}\rightarrow M_4(A_{\mu})$ by
$$
\lambda_t(u)=\psi\circ\phi(u);\qquad \lambda_t(v)=V_t.
$$
Indeed, direct calculation shows that
\begin{eqnarray*}
\|\lambda_t(x)\|&=&\left\|\begin{pmatrix}\sin^2t\cdot x&0&0&\sin
t\cos t\cdot x\\0&0&0&0\\0&0&0&0\\ \sin t\cos t\cdot
x&0&0&-\sin^2t\cdot
x\end{pmatrix}\right\|\\&=&\left\|\begin{pmatrix}\sin^2t&\sin
t\cos t\\ \sin t\cos t&-\sin^2
t\end{pmatrix}\right\|\cdot\|x\|=\sin t\cdot\|x\|\leq\|x\|\leq\mu,
\end{eqnarray*}
where $x=u+u^*+v+v^*$, hence $\lambda_t$ is continuous for any
$t\in[0,\pi/2]$.

Then $\lambda_0$ and $\lambda_{\pi/2}$ induce the same map for the
$K$-theory. At the end-points one has $\lambda_0=\psi\circ\phi$
and $\lambda_{\pi/2}=\operatorname{id}_{A_{\mu}}\oplus \tau_1
\oplus \tau_2 \oplus \tau_3$, where $\tau_1(u)=1_{A_\mu}$,
$\tau_1(v)=-1_{A_\mu}$; $\tau_2(u)=-1_{A_\mu}$,
$\tau_2(v)=1_{A_\mu}$; $\tau_3(u)=\alpha(v)$,
$\tau_3(v)=\alpha(u)$.

Let $\tau:A_{\mu}\rightarrow A_{\mu}$ be a $*$-homomorphism given
by $\tau(u)=\tau(v)=i\cdot 1_{A_\mu}$. The formulas $u_t=
-t\operatorname{Re}v+i\sqrt{1-t^2(\operatorname{Re}v)^2},v_t=
-t\operatorname{Re}u+i\sqrt{1-t^2(\operatorname{Re}u)^2}$,
$t\in[0,1]$, provide a homotopy connecting $\tau_3$ and $\tau$.
Similarly, $\tau_1$ and $\tau_2$ are homotopic to $\tau$ due to
the homotopies $u_t=(\pm\cos t+i\sin t)\cdot 1_{A_\mu}$,
$v_t=(\mp\cos t+i\sin t)\cdot 1_{A_\mu}$, $t\in [0,\pi/2]$. All
these homotopies satisfy the constraint
$\|u_t+u_t^*+v_t+v_t^*\|\leq\mu$ when $\|u+u^*+v+v^*\|\leq\mu$.

Thus, for the induced maps in $K_*$-groups one has
$(\psi\circ\phi)_*=\operatorname{id}_{K_*(A_\mu)}+3\tau_*$, or,
equivalently,
$$
\operatorname{id}_{K_*(A_\mu)}=(\psi\circ\phi)_*-3\tau_*.
$$

Note that $\tau$ factorizes through $\mathds{C}$:
$\tau:A_{\mu}\rightarrow \mathds{C}\rightarrow A_{\mu}$.
Therefore, for $K_1$, the map $\tau_*:K_1(A_\mu)\to K_1(A_\mu)$ is
zero (as $K_1(\mathds{C})=0$), so
$\operatorname{id}_{K_1(A_{\mu})}=(\psi\circ\phi)_*$.

For any $\mu\in(0,4)$, consider the commuting diagram
$$
\begin{xymatrix}{
K_1(A_4)\ar[d]\ar[rd]\ar@{=}[rr]&&K_1(A_4)\ar[d]\\
K_1(A_\mu)\ar[r]^-{\phi_*}\ar[d]&K_1(B)\ar[ur]\ar[dr]\ar[r]^-{\psi_*}&K_1(A_\mu)\ar[d]\\
K_1(A_0)\ar[ru]\ar@{=}[rr]&&K_1(A_0) }
\end{xymatrix}
$$
where the diagonal arrows are isomorphisms and the compositions of
the vertical arrows are identity maps. The latter implies that the
map $K_1(A_4)\to K_1(A_\mu)$ is injective. If it is not
surjective, there would exist some element in $K_1(A_\mu)$ that
doesn't come from $K_1(A_4)$, but this contradicts
$\operatorname{id}_{K_1(A_\mu)}=\psi_*\circ\phi_*$. Thus, the map
$K_1(A_4)\to K_1(A_\mu)$ induced by the quotient map is an
isomorphism.

As the map $\tau_*:K_0(A_\mu)\to K_0(A_\mu)$ is not trivial, the
case of $K_0$ is slightly more difficult, and we deal with it
below.

Recall that $\tau$ factorizes through $\mathds{C}$. Let
$\rho:A_\mu\to\mathds{C}$ denote the character such that
$\tau=\iota\circ\rho$, where $\iota:\mathds{C}\to A_\mu$ is the
canonical inclusion of scalars, $\iota(\lambda)=\lambda\cdot
1_{A_\mu}$. Then $\rho(u)=\rho(v)=i$.

Note that the composition $\begin{xymatrix}{K_0(\mathds{C})
\ar[r]^-{\iota_*} & K_0(A_{\mu}) \ar[r]^-{\rho_*} &
K_0(\mathds{C})}\end{xymatrix}$ is the identity map on
$K_0(\mathds{C})$. Thus $K_0(A_{\mu})=K_0(\mathds{C})\oplus \ker
\tau_*$.

Let $\sigma:B\rightarrow \mathds{C}$ be a $*$-homomorphism defined
by $\sigma((z,1))=\sigma((1,z))=i$. Then
$\sigma((-1,\alpha(z)))=\alpha(i)=i$. Therefore,
$\sigma(\phi(u))=\sigma(\phi(v))=\begin{pmatrix}i&0\\0&i\end{pmatrix}$,
hence $\sigma\circ\phi=\rho\oplus\rho$.

Let $p\in K_0(A_\mu)$, $p\in\ker\rho_*=\ker\tau_*$. Then $(\sigma
\circ \phi)_*(p)=2\rho_*(p)=0$. As $\sigma_*:K_0(B)\to
K_0(\mathds{C})$ is an isomorphism, so $p\in \ker \phi_*$.

Since $\operatorname{id}_{K_0(A_\mu)}=(\psi\circ\phi)_*-3\tau_*$,
$$
p=(\psi\circ\phi)_*(p)-3\tau_*(p)=\psi_*(\phi_*(p))-3\iota_*(\rho_*(p))=0,
$$
hence $\ker \tau_*=0$, $K_0(A_{\mu})\cong \mathds{Z}$ (generated
by $[1_{A_\mu}]$).

\end{proof}

\begin{ack*}
Part of this work was done during the stay of the second author at the Focused Semester on $KK$-Theory and its Applications held at the University of M\"{u}nster in May-June 2009. He would like to express his gratitude to the organizers, especially Prof. Echterhoff and Dr. Paravicini, for their kind hospitality.
\end{ack*}

\bibliographystyle{amsplain}

\end{document}